%% file: main.tex
\documentclass[10pt]{amsart}
\usepackage{amsmath,amssymb, amsthm} 

\theoremstyle{plain}
    \newtheorem{thm}{Theorem}
    \newtheorem{cor}[thm]{Corollary}
    \newtheorem{prop}{Proposition}[section]
    
    \newtheorem{lemma}[prop]{Lemma}

\theoremstyle{definition}
    
\theoremstyle{remark}
    \newtheorem{rem}[prop]{Remark}

\def\vol{{\operatorname{vol}}}

\def\inter{\mathrm{int}}

\def\rank{\mathrm{rank}}

\def\div{\mathrm{div}}

\newcommand{\R}{\mathbb{R}}\newcommand{\Z}{\mathbb{Z}}\newcommand{\N}{\mathbb{N}}

\renewcommand{\setminus}{\smallsetminus}
\renewcommand{\emptyset}{\varnothing}

\newcommand{\id}{\mathit{id}}

\newcommand{\comm}[1]{}

\begin{document}

\title{On the regularization of conservative maps}

\author[A.~Avila]{Artur Avila}
\address{CNRS UMR 7599,
Laboratoire de Probabilit\'es et Mod\`eles al\'eatoires.
Universit\'e Pierre et Marie Curie--Bo\^\i te courrier 188.
75252--Paris Cedex 05, France}
\curraddr{IMPA,
Estrada Dona Castorina 110, Rio de Janeiro, Brasil} 
\urladdr{www.proba.jussieu.fr/pageperso/artur/}
\email{artur@math.sunysb.edu}

\begin{abstract}
We show that smooth maps are $C^1$-dense among $C^1$ volume preserving maps.
\end{abstract}

\date{\today}

\maketitle




\input{appro}

\end{document}

%% file: appro.tex
\section{Introduction}

Let $M$ and $N$ be $C^\infty$ manifolds\footnote{All manifolds will be
assumed to be Hausdorff, paracompact and without boundary, but possibly not
compact.} and let $C^r(M,N)$ ($r \in \N \cup
\{\infty\}$) be the space of
$C^r$ maps from $M$ to $N$, endowed with the Whitney topology.
It is a well known fact that $C^\infty$ maps are dense in $C^r(M,N)$.  Such
a result is very useful in differentiable topology and
in dynamical systems (as we will discuss in more detail).  On the other
hand, in closely
related contexts, it is the non-existence of a regularization theorem that
turns out to be remarkable: if homeomorphisms could always be approximated
by diffeomorphisms then the whole theory of exotic structures would
not exist.

Palis and Pugh \cite {PP}
seem to have been the first to ask about the
corresponding regularization results in the case of conservative and
symplectic maps.  Here one fixes $C^\infty$ volume forms\footnote {For
non-orientable manifolds, a volume form should be understood up to sign.}
(in the conservative case) or symplectic structures (symplectic case), and
asks whether smoother maps in the corresponding class are dense with
respect to the induced Whitney topology.
The first result in this direction was due to Zehnder \cite {Z}, who
provided regularization theorems for symplectic maps, based on the use of
generating functions.  He also provided a
regularization theorem for conservative maps, but only when $r>1$ (he did
manage to treat also non-integer $r$).  The case $r=1$ however has
remained open since then (due in large part to intrinsic difficulties
relating to the PDE's involved in Zehnder's approach,
which we will discuss below),
except in dimension $2$, where it is equivalent to the symplectic case.
This is the problem we address in this paper.  Let $C^r_\vol(M,N) \subset
C^r(M,N)$ be the subset of maps that preserve the fixed smooth
volume forms.

\begin{thm} \label {dens}

$C^\infty$ maps are dense in $C^1_\vol(M,N)$.

\end{thm}

Let us point out that the corresponding regularization theorem for
conservative flows was obtained much earlier by Zuppa \cite {Zup}
in 1979.  In fact, in
a more recent approach of Arbieto-Matheus \cite {AM}, it is shown that
a result of Dacorogna-Moser \cite {DM}
allows one to reduce to a local situation where
the regularization of vector fields which are divergence free can be treated
by convolutions.  However,
attempts to
reduce the case of maps to the case of flows through a suspension
construction have not been succesful.

Let us discuss a bit an approach to this problem which is succesful in
higher regularity,
and the difficulties that appear when considering $C^1$ conservative
maps.  Let us assume for simplicity that $M$ and $N$ are compact, as all
difficulties are already present in this case.
Let $f \in C^r_\vol(M,N)$, and let $\omega_M$ and $\omega_N$ be the
smooth volume forms.  Approximate $f$ by a smooth
non-conservative map $\tilde f$.  Then $\tilde f^* \omega_N$ is $C^{r-1}$
close to $\omega_M$.  If we can solve the equation $h^* \tilde f^*
\omega_N=\omega_M$ with $h$ $C^r$ close to $\id$ then the desired
approximation could be obtained by taking $\tilde f \circ h$.  Looking at
the local problem one must solve to get $h$, it is
natural to turn our attention to the $C^r$ solutions of the
equation $\det Dh=\phi$ where $\phi:\R^n \to \R$
is smooth and close to $1$.

Unfortunately, though $\phi$ is smooth, we only know apriori that the
$C^{r-1}$ norm of $\phi$ is small.  This turns out to be quite sufficient to
get control on $h$ if $r \geq 2$, according to the Dacorogna-Moser
technique.  But when $r=1$, the analysis of the equation is different, as
was shown by Burago-Kleiner \cite {BK} and McMullen \cite {McM}.
This is well expressed in the following result, Theorem 1.2 of \cite {BK}:
{\it Given $c>0$ there exists a continuous function
$\phi:[0,1]^2 \to [1,1+c]$,
such that there is no bi-Lipschitz map $h:[0,1]^2 \to \R^2$ with
$\det Dh=\phi$.}

This implies that continuous volume forms on a
$C^\infty$ manifold need not be $C^1$ equivalent to smooth volume forms.
This is in contrast with the fact that all smooth volume forms are $C^\infty$
equivalent up to scaling \cite {M}, and the differential topology
fact that all $C^1$ structures on a $C^\infty$ manifold
are $C^1$ equivalent.

\begin{rem}

One can define a $C^r_\vol$ structure on a manifold as a maximal atlas whose
chart transitions are $C^r$ maps preserving the usual volume of $\R^n$ (see
\cite {T}, Example 3.1.12).
Then Theorem \ref {dens} (and its equivalent for higher differentiability
\cite {Z}) can be used to conclude that any $C^r_\vol$ structure is
compatible with a $C^\infty_\vol$ structure (unique up to
$C^\infty_\vol$-diffeomorphism by \cite {M}), by
following the proof of the corresponding statement for $C^r$-structures (see
Theorem 2.9 of \cite {H}).  For $r \geq 2$, a $C^r_\vol$ structure
is the same as a $C^r$ structure together with a $C^{r-1}$ volume form by
\cite {DM},
but not all continuous volume forms on a $C^\infty$ manifold arise from a
$C^1_\vol$ structure, by Theorem 1.2 of \cite {BK} quoted above.

We notice also the following amusing consequence of Theorem 1.2 of \cite
{BK}, which we leave as an exercise:
{\it A generic continuous volume form on a $C^1$ surface has no
non-trivial symmetries, that is,
the identity is the only diffeomorphism of the
surface preserving the volume form.}  This highlights that the correct
framework to do $C^1$ conservative dynamics is the $C^1_\vol$ category (and
not $C^1$ plus continuous volume form category).

\end{rem}

The equation $\det Dh=\phi$ has been studied also in other regularity
classes (such as Sobolev) by Ye \cite {Y} and Rivi\`ere-Ye \cite {RY},
but this has not helped with
the regularization theorem in the $C^1$ case.


The approach taken in this paper is very simple, ultimately constructing
a smooth approximation by taking independent linear approximations
(derivative) in a very dense set, and carefully modifying and gluing them
into a global map (with a mixture of bare-hands technique and some
results from the PDE approach {\it in high regularity}).  A key point is to
enforce that the choices involved in the construction are made through a
{\it local} decision process.  This is useful to avoid
long-range effects, which if left out would lead us to a
discretized version of the PDE approach in low regularity, with the
associated difficulties.  To ensure locality, we use the original
unregularized map $f$ as {\it background data} for making the decisions.
The actual details of the procedure are best understood by going through the
proof, since the difficulties of this problem lie in the details.

\subsection{Dynamical motivation}

In the discussion below, we restrict ourselves to diffeomorphisms of compact
manifolds for definiteness.

There is a good reason why the regularization problem for conservative maps
has been first introduced in a dynamical context.  In dynamics, low
regularity is often used in order to be able to have available the strongest
perturbation results, such as the Closing Lemma \cite {P}, the
Connecting Lemma \cite {Ha} and the simple but widely used
Franks' Lemma \cite {F}.  Currently such results are
only proved precisely for the
$C^1$ topology (even getting to $C^{1+\alpha}$ would be an amazing
progress), except when considering one-dimensional dynamics.  On the other
hand higher regularity plays a fundamental role when distortion needs to be
controlled, which is the case for instance when
the ergodic theory of the maps is the focus
($C^{1+\alpha}$ is a basic hypothesis of
Pesin theory, and for most
results on stable ergodicity such as \cite {BW}, though more
regularity is necessary if KAM methods are involved \cite {RH}).
While dynamics in the smooth and the low regularity worlds may often seem to
be different altogether (compare \cite {BV} and \cite {BV1}),
it turns out that their characteristics can be often combined
(both in the conservative and the dissipative setting), yielding for
instance great flexibility in obtaining interesting examples: see the
construction of non-uniformly hyperbolic Bernoulli maps \cite {DP} which
uses $C^1$-perturbation techniques of \cite {B}.

In the dissipative and symplectic settings, regularization theorems have
been an important
tool in the analysis of $C^1$-generic dynamics: for instance, Zehnder's
Theorem is used in the proof of \cite {ABW} that ergodicity is $C^1$-generic
for partially hyperbolic
symplectic diffeomorphisms.
\footnote {It was this fact
indeed that convinced the author to work on Theorem \ref {dens}.}
Thus it is natural to expect that Theorem \ref {dens} will lead to
several applications on $C^1$-generic conservative dynamics.  Indeed many
recent results have been stated about certain properties of
$C^2$-maps being dense in the $C^1$-topology, without being able to conclude
anything about $C^1$-maps only due to the non-availability of Theorem \ref
{dens}.  Thus it had been understood for some time that proving
Theorem \ref {dens} would have many immediate applications.
Just staying with examples in the line of \cite {ABW}, we point out that
\cite {BMVW} now implies that ergodicity is $C^1$-generic for partially
hyperbolic maps with one-dimensional center
(see section 4 of \cite {BMVW}), and
the same applies to the case of two-dimensional center, in view of the recent
work \cite {RRTU}.

Though we do not aim to be exhaustive in the discussion of applications
here, we give a few other examples which were pointed out to us
by Bochi and Viana:

\noindent 1.
Any $C^1_\vol$-robustly transitive
diffeomorphism admits a dominated splitting (conjectured, e.g., in
\cite {BDP}, page 365), a result obtained for $C^{1+\alpha}$
diffeomorphisms in \cite {AM} using a Pasting Lemma.  (We note that
this work also allows one to extend the
Pasting Lemma of \cite {AM} itself, and hence its other consequences,
to the $C^1$ case.)

\noindent 2.
A $C^1$-generic consevative non-Anosov diffeomorphism has only
hyperbolic sets of zero Lebesgue measure.
Zehnder's Theorem has been used in \cite {B} and \cite {BV1} to
achieve this conclusion in the symplectic case, and such a result is
necessary for the conclusion of
the central dychotomy of \cite {B}.  It is based on a statement
about $C^2$ conservative maps obtained in \cite {BV2}, so the conclusion for
conservative maps now follows directly from Theorem \ref {dens}.
We hope that results in this direction will play a role in further
strengthenings of \cite {BV1}.

\noindent 3.
The existence of locally generic non-uniformly hyperbolic ergodic
conservative diffeomorphisms with non-simple Lyapunov spectrum \cite {BonV},
\cite {Tz} (the proof, conditional to the existence of regularization, is
nicely sketched in page 260 of \cite {BDV}).

\subsection{Outline of the proof}

Our basic idea is to construct the approximation of a diffeomorphism ``from
inside'', growing it up through a growing frame while paying attention to
compatibilities.

Let us think first of the case where we have a $C^1_\vol$ map
$f:\R^n \to \R^n$, whose derivative is bounded and uniformly continuous.  We
wish to approximate $f$ by a $C^\infty_\vol$ map, $C^1$-uniformly.  We
break $\R^n$ into small cubes with vertices in a multiple of the lattice
$\Z^n$.  In each cube, the derivative of $f$ varies little.  Thus $f$
restricted to each cube admits certainly a nice $C^\infty$ approximation: in
fact, we can just approximate it by an affine map.  Annoyingly, those
approximations do not match.

Our next attempt is to build the approximation more slowly.  First we will
construct an approximation in a neighborhood of the set of vertices of the
cubes, then extend it to an approximation near the set of edges, etc.
Progressing through the $k$-faces of the cubes, $0 \leq k \leq n$, we will
eventially get a map defined everywhere.

The first step is easy: consider a small $\epsilon$-neighborhood $V_0$
of the set
of vertices of all the cubes.  In this set, we can define an approximation
$f_0$ of $f$ which is just affine in each connected component.  Next,
consider a
$\epsilon^2$-neighborhood $V_1$ of the set of edges.  The connected
components of $V_1 \setminus V_0$ intersect, each, a single edge.  We can
extend $f_0|V_0 \cap V_1$ to a map $f_1$ defined in $V_1$: the extension
argument follows \cite {DM}, and is based on the fact that $f_0$ admits a
nice $C^\infty$ (apriori non-volume preserving) extension.  This extension,
behaves well at the scale of the cubes (after rescaling to unit size, the
extension is $C^\infty$-close to affine), which yields the estimates
necessary to apply the (high regularity) Dacorogna-Moser argument.

We repeat this process until getting a map $f_{n-1}$ defined in a
neighborhood $V_{n-1}$ of the faces
of the cubes.  It is important to emphasize that, along this process, all
decisions taken are local: for instance, to know what to do near an edge we
only need to look at what we have done near the vertices of this edge.  This
eliminates long range effects in the process.

At the last moment however, we face a new
difficulty: there is an obstruction to the extension of $f_{n-1}$ to a
volume preserving map.  In fact, $\R^n \setminus V_{n-1}$ is disconnected,
and for an extension (close to $f$)
to exist, the boundary $P$ (topological sphere)
of each hole must be mapped,
under $f_{n-1}$ to a topological sphere $P'$ such that the bounded
components of $\R^n \setminus P$ and $\R^n \setminus P'$ have the same
volume.

To account for this, one could try to modify the map $f_{n-1}$, so that the
volume of the ``holes in the image'' is the same as the volume of the
``holes in the domain''.  In fact, if
we have a volume preserving map such as $f_{n-1}$, defined in a neighborhood
of the boundaries
of the cubes, it is easy to modify it to ``shift mass'' between adjacent
``holes in the image''.  We could try to correct an increasing family of
holes: choose one hole and an adjacent one, move mass so that the first one
becomes fine, then choose another adjacent hole to the second one, move
mass, etc.  But this introduces possible long range effects:
the decisions taken early on, in some specific place,
affect what we have to do much later, and far away.  Thus it is better to
try to do it simultaneously.  How to
prescribe how much mass should be moved from which hole to which hole? 
Trying to solve this takes as to some difference equation: we are given a
function $d$ from the set of cubes to $\R$ (measuring the excess or
deficit of volume of the ``hole in the image''), and we want to find some
function $s$
from the set of faces to $\R$ such that the sum of $s$ over all faces of
each cube equals $d$.  This is just some discretized form of the divergence
equation, and we do not want to follow this path, since, as described
before, the divergence equation is hard to solve in the regularity we are
dealing with.

We will instead proceed differently, being careful to make the
constructions of $f_0$,...,$f_{n-1}$, so that the problem will not show up
at the last step: we make the corrections along the way, which breaks the
problem into simple ones (we want to be able to make local decisions).
To make the decisions, we use an important guide: the
``background'' map $f$, which is known to be volume preserving.
When constructing $f_0$, we make
sure that $f_0$, near each vertex $p$, is fair to all cubes $C$
that have $p$ as a
vertex: thus if $B$ is the connected component of $V_0$ containing $p$,
we want that $f_0(B \cap C)$ and $f(B \cap C)$ have the same volume.  This
can be done, starting with a careless attempt at defining $f_0$, such as the
one considered before, by a ``moving mass'' argument, which this time
has no long range effects.  Later, when defining $f_1$ near an edge
$q$, the fairness property of $f_0$
will allow us to be fair to all cubes that have $q$ as an edge.  This goes
on until $f_{n-1}$, when we find out that the fairness condition implies
that there is no problem with the holes any more.  We can then extend
$f_{n-1}$ to the desired approximation $f_n$ of $f$.

This concludes the argument in this case.  We can adapt this argument to
deal with, instead of the entire $\R^n$, some domain in $\R^n$.  We just
need to consider a suitable decomposition into cubes which has locally
bounded geometry, and the Whitney decomposition will do.  In fact, we can
prove a more detailed result about domains, with ``matching conditions''
(thus, if $f$ is already smooth somewhere, we do not need to modify $f$
there along the approximation\footnote {We note that this kind of result
is more relevant for ``Pasting Lemma'' applications \cite {AM} than Theorem
\ref {dens} itself.}).
Once the case of
domains in $\R^n$ is taken care of,
we can deal with the case of manifolds
as well by a triangulation argument, building the approximation through
vertices, edges, etc., but with a much easier argument (since we can
prescribe matching conditions).

This paper is organized as follows.
We first describe the kind of extension
result we will repeatedly make use of, obtained using the Dacorogna-Moser
technique.  Then we show how to move mass between cubes, to
achieve fairness.  Next, we formulate and prove
a version of the approximation
theorem with matching conditions.  We conclude with the application of this
result to the case of manifolds.

{\bf Aknowledgements:}  I am grateful to Jairo Bochi for
stimulating conversations while this work was under way, to Carlos Matheus
for confirming the proof, and to Federico Rodriguez-Hertz for explaining to
me what makes the immediate ``reduction to the case of flows'' not so
straightforward.  The writing of this note was significantly improved thanks
to detailed comments provided by Amie Wilkinson.
This research was partially conducted during the period
A.A. served as a Clay Research Fellow.

\section{Extending conservative maps}

Fix two connected open sets with smooth boundary
$B_1,B_2 \subset \R^n$ with $\overline B_1 \subset B_2$ and $\overline B_2
\setminus B_1$ smoothly
diffeomorphic to $\partial B_1 \times [0,1]$.
For the proof of Theorem \ref {dens}, we will need the following
slight variations of Theorems 2  and 1
of Dacorogna-Moser \cite {DM}.

\begin{thm}

Let $\phi:\R^n \to \R$ be a $C^\infty$ function with $\int \phi=0$ supported
inside $B_1$.
Then there exists $v \in C^\infty(\R^n,\R^n)$ supported inside $B_2$
with $\div v=\phi$.
Moreover, if $\phi$ is $C^\infty$ small then
$v$ is $C^\infty$ small.

\end{thm}

\begin{proof}

Theorem 2 of \cite {DM} states in a more general context,
that there exists $w:\overline B_2 \to \R^n$ with
$\div w=\phi$ and $w|\partial B_2=0$, and if $\phi$ is $C^\infty$ small then
$w$ is also $C^\infty$ small.  It is thus
enough to find some $C^\infty$ $u:\overline B_2 \to \R^n$ (small if $\phi$
is small) with $\div u=0$ and $u|\overline B_2 \setminus B_1=w$, and let
$v|\overline B_2=w-u$, $v|\R^n \setminus \overline B_2=0$.  This procedure
is the standard one already used in \cite {DM}.

There is a duality between smooth vector fields $u$ and smooth
$n-1$-forms $u^*$, given by
$u^*(x)(y_1,...,y_{n-1})=\det(u(x),y_1,...,y_{n-1})$.
The duality transforms the equation $\div u=0$ into $du^*=0$.  The form
$w^*$ is thus closed in $\overline B_2 \setminus B_1$, and the boundary
condition $w|\partial B_2=0$ implies that it is exact in $\overline B_2
\setminus B_1$.  Solve the equation $d\alpha=w^*$ in
$\overline B_2 \setminus B_1$ and extend $\alpha$ smoothly to
$\overline B_2$ (notice that $\alpha$ can be required to be small
if $w$ is small).  Let $u$ be a vector
field on $\overline B_2$ given by $d \alpha=u^*$.
Then $u|\overline B_2 \setminus B_1=w$, and since $d u^*=0$ in
$\overline B_2 \setminus B_1$, we have $\div u|\overline B_2 \setminus
B_1=0$.
\end{proof}


\begin{thm}

Let $\phi:\R^n \to \R$ be a $C^\infty$ function with $\int \phi=0$
supported inside $B_1$.
Then there exists $\psi \in C^\infty(\R^n,\R^n)$ with $\psi-\id$
supported inside $B_1$ such that $\det \psi=1+\phi$.
Moreover, if $\phi$ is $C^\infty$-small then
$\psi-\id$ is $C^\infty$ small.

\end{thm}

\begin{proof}

As in \cite {DM}, the solution is
given explicitly as $\psi=\psi_1$ where
$\psi_t(x)$ is the solution of the differential equation $\frac {d} {dt}
\psi_t(x)=v(\phi_t(x))/(t+(1-t) (1+\phi(\psi_t(x))))$
with $\psi_0(x)=x$ and $v$ comes from the previous theorem.
\end{proof}

\begin{cor} \label {infty}

Let $K$ be a compact set, $U$ be a neighborhood of $K$ and let
$f \in C^\infty(\R^n,\R^n)$ be $C^\infty$ close to the identity and
such that $f|U$ is volume preserving.  Assume that for every
bounded connected component $W$ of $\R^n \setminus K$, $W$ and
$f(W)$ have the same volume.  Then there exists
a $C^\infty$ conservative map close to the identity such that
$\tilde f=f$ on $K$.

\end{cor}

\begin{proof}

We may modify $f$ away from $K$ so that $f-\id$ is compactly supported and
$\det f-1$ is supported inside some open set $\tilde B_1$, which can be
assumed to have smooth boundary, disjoint from some neighborhood of $K$.
Let $m$ be the number of connected components of $\tilde B_1$.
We can assume that each connected component of $\R^n \setminus K$
contains at most one connected component of
$\tilde B_1$ (otherwise we just enlarge $\tilde B_1$ suitably).
For each connected component
$B^i_1$ of $\tilde B_1$, select a small
$\epsilon$-neighborhood $B^2_i$ of $\overline B^i_1$.  Let $\phi_i$ be
given by $\phi_i|B_1^i=\det f-1$ and $\phi|\R^n \setminus B_1^i=0$.  Then
$\int \phi_i=0$.  Indeed,
if $B_1^i$ is contained in a bounded connected component of
$\R^n \setminus K$, this follows immediately
from $f$ preserving the volumes of such sets, and if $B_1^i$ is contained in
the unbounded component $W$ of $\R^n \setminus K$, one uses that $f$
preserves the volume of $W \cap B$ for all sufficiently large balls $B$ (to
see this one uses that $f-\id$ is compactly supported).  Applying the
previous theorem, one gets maps $\psi_i$ with $\psi_i-\id$ supported inside
$B_1^i$.  We then take $\tilde f=f \circ \psi_1^{-1} \circ \cdots \circ
\psi_m^{-1}$.
\end{proof}

\section{Moving mass}

In this section we will consider the $L^\infty$ norm in $\R^n$.  The closed
ball of radius $r>0$ around $p \in \R^n$ will be denoted by $B(p,r)$
(this ball is actually a cube).  The canonical basis of $\R^n$ will be
denoted by $e_1,...,e_n$.

\begin{lemma} \label {F}

Fix $0<\delta<1/10$.
Let $S \subset \{1,...,n\}$ be a subset with $0 \leq k \leq n-1$ elements. 
Let $P \subset \R^n$ be the (finite)
set of all $p$ of the form $\sum_{i \notin S} u_i
e_i$ with $u_i=\pm 1$.  Let $B=\cup_{p \in P} B(p,1)$, and let $B'$ be the
open $\delta$-neighborhood of $B$.  Let $W$ be a
Borelian set whose $\delta$-neighborhood is contained in $B$ and which
contains $B(0,\delta)$.  If $F \in
C^1_\vol(B',\R^n)$ is $C^1$ close to the identity,
then there exists $s \in C^\infty_\vol(\R^n,\R^n)$ such that
\begin{enumerate}
\item $s|\inter B(0,10)$ is $C^\infty$ close to the identity,
\item $\vol F(s(W)) \cap B(p,1)=\vol W \cap B(p,1)$ for $p \in P$,
\item $s$ is the identity outside the $\delta$-neighborhood the subspace
generated by the $\{e_i\}_{i \notin S}$.
\end{enumerate}

\end{lemma}

\begin{proof}

Notice that there are $2^{n-k}$ elements in $P$.
Call two elements $p,p' \in P$ adjacent if $p-p'=\pm 2 e_l$
for some $1 \leq l \leq n$.

Let $p,p'$ be adjacent.
Let $q=q(p,p')=\delta (p'+p)/4$, and let $C=C(p,p')$ be the cylinder
consisting of all
$z \in \R^n$ of the form $z+t(p'-p)$ where $t \in \R$ and $z \in
B(q,\delta/4)$.  Let $\phi=\phi^{p,p'}:\R^n \to [0,1]$ be a
$C^\infty$ function such that $\phi(q)=1$,
$\phi|\R^n \setminus C=0$ and
$\phi(z+t (p-p'))=\phi(z)$ for $t \in \R$.  For $z \in \R$, let
$s_t=s_t^{p,p'} \in C^\infty_\vol(\R^n,\R^n)$ be given by
$s_t(z)=z+t \phi(z) (p-p')$.

Let us show that for $|t|<\delta/100$, we have
\begin{align*}
\vol s_t(W) \cap B(p,1)-\vol &W \cap B(p,1)\\
&=\vol s_t(C \cap B(0,\delta)) \cap
B(p,1)-\vol C \cap B(0,\delta) \cap B(p,1)\\
&=t \int_{B(0,1/2)} \phi(z) dz.
\end{align*}
Indeed, since the $\delta$-neighborhood of $W$ is contained in $B$, and
$s_t$ is the identity outside $C$, if $z \in W \cap B(p,1)$ belongs
(respectively, does not belong)
to $C \cap B(0,\delta)$ then $s_t(z)$ belongs (respectively, does not
belong) to $B(p,1)$ as well.  Since $C \cap B(0,\delta) \subset W$, this
justifies the first equality.  The second equality is a straightforward
computation.

Let $B''$ be the $\delta/2$ open neighborhood of $B$.
It is easy to see that if
$\tilde F \in C^1_\vol(B'',\R^n)$ is $C^1$ close to the
identity then $\vol \tilde F(s_t(W)) \cap
B(\tilde p,1)=\vol \tilde F(W) \cap B(\tilde p,1)$ for every $\tilde p \in P
\setminus \{p,p'\}$, since in this case we actually have
$\tilde F(s_t(W)) \cap B(\tilde p,1)=\tilde F(W) \cap B(\tilde p,1)$.

We claim that there exists $t \in \R$ small such that
$\vol \tilde F(s_t(W)) \cap B(p,1)=\vol W \cap B(p,1)$.
Indeed, for $|t|<\delta/100$,
$\vol \tilde F(s_t(W)) \cap B(p,1)-\vol s_t(W) \cap B(p,1)$ is small if
$\tilde F$ is close to the identity.  For $|t|<\delta/100$, we can directly
compute $\vol s_t(W) \cap B(p,1)-\vol W \cap
B(p,1)=t \int_{B(0,1/4)} \phi(z) dz$.  Thus the claim
follows from the obvious continuity of $t \mapsto \vol \tilde F(s_t(W)) \cap
B(p,1)$ for $|t|<\delta/100$.

As a graph, $P$ is just an hypercube, so there exists
an ordering $p_1,...,p_{2^{n-k}}$ of the elements of $P$ such
that for $1 \leq i \leq 2^{n-k}-1$,
$p_i$ and $p_{i+1}$ are adjacent.  Given $F$, we define sequences
$F_{(l)} \in C^1_\vol(B'',\R^n)$,
$s_{(l)} \in C^\infty_\vol(\R^n,\R^n)$,
$0 \leq l \leq 2^{n-k}-1$ by induction as follows.  We let
$s_{(0)}=\id$, $F_{(l)}=F \circ s_{(l)}$ for $0 \leq l \leq 2^{n-k}-1$, and
for $1 \leq l \leq 2^{n-k}-1$ we let $s_{(l)}=s^{p,p'}_t
\circ s_{(l-1)}$ for $1 \leq l \leq
2^{n-k}-1$, where $p=p_i$, $p'=p_{i+1}$, and $t$
is given by the claim applied with $\tilde F=F_{(l-1)}$.  As long as $F$ is
sufficiently close to the identity, we get inductively that $F_{(l)}$ is
close to the identity, so this construction can indeed be carried out.

Let us show that $s=s_{(2^{n-k}-1)}$ has all the required properties.
Properties (1) and (3) are rather clear.  By construction, we get
inductively that
$\vol F_{(l)}(W) \cap B(p,1)=\vol W \cap B(p,1)$ for $p \in
\{p_1,...,p_l\}$, so it is clear that $\vol F(s(W)) \cap B(p,1)=\vol W \cap
B(p,1)$ except possibly for $p=p_{2^{n-k}}$.  But $\sum_{p \in P}
\vol F(s(W)) \cap B(p,1)=\vol F(s(W)) \cap B=\vol W \cap B=\sum_{p \in P}
\vol W \cap B(p,1)$, so we must have $\vol F(s(W)) \cap B(p,1)=\vol W \cap
B(p,1)$ also for $p=p_{2^{n-k}}$, and property (2) follows.
\end{proof}

\section{Proof of Theorem \ref {dens}}

\subsection{Charts}

If $U \subset \R^n$ is open and
$f:U \to \R^n$ is a bounded $C^r$ map with bounded derivatives up to order
$r$, we let $\|f\|_{C^r}$ be the natural $C^r$ norm.

\begin{thm} \label {W}

Let $W$ be an open subset of $\R^n$ and let $f \in C^1_\vol(W,\R^n)$
be a map with bounded uniformly continuous derivative.  Let $K_0 \subset W$
be a compact set such that $f$ is $C^\infty$ in a neighborhood of $K_0$. 
Let $U \subset W$ be open.  Then
for every $\epsilon>0$ there exists $\tilde f \in C^1_\vol(W,\R^n)$
such that $\tilde f|U$ is $C^\infty$, $\tilde f$ coincides with $f$ in $W
\setminus U$ and in a neighborhood of $K_0$, and $\|f-\tilde
f\|_{C^1}<\epsilon$.

\end{thm}

\begin{proof}

We will consider the $L^\infty$ metric in $\R^n$.  Let $\theta>0$ be such
that the $\theta$-neighborhood of $K_0$ is contained in $W$ and $f$ is
$C^\infty$ in it.  We will now introduce a Whitney decomposition of $U$.

If $0 \leq m \leq n$,
An $m$-cell $x$ is some set of the form
$\prod_{k=1}^n [2^{-t} a_k,2^{-t} (a_k+b_k)]$ where
$t \in \Z$, $a_k
\in \Z$ and $b_k \in \{0,1\}$ with $\#\{b_k=1\}=m$.  For $m \geq 1$,
we let its interior $\inter x$ be
$\prod_{k=1}^n (2^{-t} a_k,2^{-t} (a_k+b_k))$, while for $m=0$ we let
$\inter x=x$.  Let $\partial x$ be $x \setminus \inter x$.

We say that an $n$-cell $x$ is $\epsilon$-small if its diameter is at most
$\epsilon$, and every $n$-cell of the
same diameter as $x$ which intersects $x$ is contained in $U$.
We say that a dyadic $n$-cell is $\epsilon$-good if it is a maximal (with
respect to inclusion) $\epsilon$-small $n$-cell.  We say that a dyadic
$m$-cell, $0 \leq m \leq n-1$,
is $\epsilon$-good if it is the intersection of all $\epsilon$-good
$n$-cells that intersect its interior.

By construction, the interiors of distinct
cells are always disjoint, and their union covers $U$.  This is what we
meant by a Whitney decomposition of $U$.

Given $\epsilon>0$,
we say that an $\epsilon$-good $m$-cell $x$ has rank $t=t(x)$ if the
minimal diameter of the
$\epsilon$-good $n$-cells containing it is $2^{-t}$ (if $m>0$, $2^{-t}$
is just the diameter of $x$).  The rank is designed
to give a measure of the
intrinsic scale of the $\epsilon$-good cells near $x$, so the more
cumbersome definition is needed to be meaningful for $m=0$.
Notice that if $x,y$ are $\epsilon$-good
cells and $x \cap y \neq \emptyset$ then $|t(x)-t(y)| \leq 1$ (otherwise
either $x$ or $y$ would not satisfy the maximality requirement of an
$\epsilon$-good $m$-cell).  Each
$\epsilon$-good $m$-cell $x$ is contained in $2^{n-m}$ $n$-cells of diameter
$2^{-t(x)}$, called neighbors of $x$ (which are not necessarily
$\epsilon$-good).

Fix some small $\epsilon>0$.  From now on, by $m$-cell we will understand an
$\epsilon$-good $m$-cell.
Let $N_m$ be the set of $m$-cells.
By construction, the interiors of distinct
cells are always disjoint, and their union covers $U$.  This is what we
meant by a Whitney decomposition of $U$.
The local geometry of the Whitney decomposition has some bounded complexity
(depending on the dimension): there exists
$C_0=C_0(n)$ such that each $m$-cell contains at most $C_0$ $k$-cells, $0
\leq k \leq m$.
Moreover, each $x \in N_m$ is the union of the interior of the $k$-cells, $0
\leq k \leq m$, contained in $x$.

For $x \in N_m$, let $D(x)$ be the
$2^{-10 (m+1)} 2^{-t(x)}$
neighborhood of $x$, and let $I(x)=\cup D(y)$ where the
union is taken over all proper subcells $y \subset x$.  Thus $I(x)$ is a
neighborhood of the boundary of $x$.
Let $B(x)=D(x) \cup I(x)$ (a somewhat larger neighborhood of $x$)
and $J(x)=D(x) \setminus I(x)$ (thus $J(x)$ is obtained by truncating
a neighborhood of $x$ near the boundary of $x$).
Notice that if $x$
and $y$ are distinct cells, $J(x)$ and $J(y)$ are disjoint.  Let $R(x)$ be
the interior of the union of all $n$-cells intersecting $x$.  Thus $R(x)$
is again a
neighborhood of $x$, larger than $B(x)$ and $D(x)$.  Notice that
the $2^{-100 (m+1)} 2^{-t(x)}$-neighborhood of
$J(x)$ contained in the interior of the union of the neighbors of $x$.

For a cell $x$, let $b$ be its baricenter, let $\lambda_x:\R^n \to \R^n$ be
given by $\lambda_x(z)=b+2^{-t(x)+1} z$, and let $H_x(z)=f(b)+Df(b)(z-b)$.
We say that $h \in C^\infty_\vol(R(x),\R^n)$ is $x$-nice if
$h-f$ is $C^1$-small,
$(\lambda_x^{-1} \circ h \circ \lambda_x)-(\lambda_x^{-1} \circ
H_x \circ \lambda_x)$ is $C^\infty$ small, and for every neighbor $y$ of
$x$, $\vol f^{-1}(h(J(x))) \cap y=\vol J(x) \cap y$.  Notice that this last
condition implies that for every $y
\in N_n$ containing $x$, $\vol f^{-1}(h(J(x))) \cap y=\vol J(x) \cap y$.

A family $\{h_x\}_{x \in N_m}$ is said to be nice if each
$h_x$ is $x$-nice and
$\|h_x-f\|_{C^1} \to 0$ uniformly as $\rank(x) \to \infty$.
Let $\rho=2^{-100 n}$.  We will now construct inductively
nice families $\{\tilde h_x\}_{x \in N_m}$, $0 \leq m \leq
n$ such that $\tilde h_x=\tilde h_y$ in a
$2^{-t(x)} \rho^{m+1}$-neighborhood of $B(y)$ whenever $y$ is a subcell of
$x$, and such that if $x$ is
$2^{-(m+1)} \theta$-close to $K_0$ then $\tilde h_x=f|R(x)$.

Let $x \in N_0$.  If $\epsilon$ is small and $x$ is $\theta/2$-close to
$K_0$, then $\tilde h_x=f|R(x)$ is $x$-nice.
Otherwise, if $\epsilon$ is sufficiently small, then by
a small $C^1$ modification of $H_x$ we obtain a map
$\tilde h_x$ which is $x$-nice.
The easiest way to see this is to first
conjugate by $\lambda_x$, bringing things to unit scale.  More precisely, we
get into the setting of Lemma \ref {F} (with $k=0$, hence $S=\emptyset$)
by putting $F=\lambda_x^{-1} \circ f^{-1} \circ H_x \circ \lambda_x$ and
$W=\lambda_x^{-1}(J(x))$.  Let $s$ be the map given by the Lemma \ref {F}.
Then $\tilde h_x=H_x \circ \lambda_x \circ s \circ \lambda_x^{-1}$
is $x$-nice.  Moreover, $\{\tilde h_x\}_{x \in N_0}$ is a nice family
since the estimates improve as the rank grows (indeed, as the rank grows,
one looks at smaller and smaller scales, and the derivative varies less and
less).

Let now $1 \leq m \leq n-1$ and assume that for every $k \leq m-1$
we have defined a nice family $\{\tilde h_x\}_{x \in N_k}$ with the required
compatibilities.

If $x \in N_m$ intersects a
$2^{-(m+1)} \theta$-neighborhood of $K_0$
just take $\tilde h_x=f|R(x)$ as
definition and it will satisfy the other compatibility by hypothesis.
Otherwise, let $Q$ be the open $\rho^m$-neighborhood of
$B(y)$ and define a map $h_x \in C^\infty_\vol(Q,\R^n)$
such that $h_x=\tilde h_y$ in the
$\rho^m$-neighborhood of $B(y)$ for every subcell
$y \subset x$.  Restricting $h_x$ to the $\rho^m/2$ neighborhood of $I(x)$,
which is a full compact set (that is, it does not disconnect $\R^n$),
since $m \leq n-1$,
and extending it to $R(x)$ using Corollary \ref {infty},
we get $h^{(1)}_x \in C^\infty_\vol(R(x),\R^n)$
which is $C^\infty$ close to $H_x$ after rescaling by $\lambda_x$.
By a $C^1$-small modification of $h^{(1)}(x)$
outside the $\rho$-neighborhood of $I(x)$, we can
obtain a nice family $\{\tilde h_x\}_{x \in N_m}$.  This is an
application of Lemma \ref {F} (with $k=m$)
analogous to the one described before.  This time, we let
$F=\lambda_x^{-1}
\circ f^{-1} \circ h^{(1)}_x \circ \lambda_x$ and $W=\lambda_x^{-1}(J(x))$.
We choose $S$ as the subset of the canonical basis of $\R^n$ which spans the
tangent space to $x$ at some (any) interior point.  Letting $s$ be the map
given by Lemma \ref {F}, the desired maps are given by
$\tilde h_x=h^{(1)}_x \circ \lambda_x \circ s \circ
\lambda_x^{-1}$.

By induction, we can construct the nice families as above for $0 \leq m \leq
n-1$.  Let now $x \in N_n$.  As before, when $x$ is close to $K_0$ the
definition is forced and there is no problem of compatibility by hypothesis.
Otherwise, let $Q$ be the open $\rho^n$-neighborhood of $I(x)$. 
As before, define a map $h_x:Q \to \R^n$ by gluing the definitions of
$\tilde h_y$ for subcells of $x$.  Notice that
$\R^n \setminus I(x)$ has two
connected components, and the bounded one is contained in $x$.  By
construction, $I(x)$ is the disjoint union of the $J(y)$ contained in it. 
Thus the volumes of $h_x(I(x)) \cap f(x)$ and $I(x) \cap x$ are equal.  This
implies that the bounded component of $\R^n \setminus h_x(I(x))$ has the
same volume as the bounded component of $I(x)$.  We can restrict $h(x)$
to the $\rho^n/2$ neighborhood of $I(x)$
and extend it to a map $\tilde h_x \in C^\infty_\vol(R(x),\R^n)$
which is $x$-nice using Corollary \ref {infty}
(after rescaling by $\lambda_x$ and
then rescaling back).
Thus we obtain a nice family $\{\tilde h_x\}_{x \in N_n}$ with all
the compatibilities.

The nice family $\{\tilde h_x\}_{x \in N_n}$ is such that whenever two
$n$-cells $x$ and $y$ intersect we have $\tilde h_x=\tilde h_y$ in a
neighborhood of the intersection.
Let $\tilde f:W \to \R^n$ be defined by $\tilde f(z)=f(z)$, $z \notin U$ and
$\tilde f(z)=\tilde h_x(z)$ for
every $z \in x$, $x \in N_n$.  Then $\tilde f \in C^1_\vol(W,\R^n)$, since
near $\partial U \cap W$ the rank of a $n$-cell $x$
is big and hence $\|\tilde h_x-f|R(x)\|_{C^1}$ is small.  Moreover
$\|\tilde f-f\|_{C^1}$ is small everywhere, and $\tilde f=f$
in a neighborhood of $K_0$ by construction.
\end{proof}

\subsection{Manifolds}

We now conclude the proof of Theorem \ref {dens} by a triangulation
argument.  Triangulate $M$ so that for every simplex $D$ there are
smooth charts $g_i:W_i \to \R^n$, $\tilde
g_i:\tilde W_i \to \R^n$ such that $f(W_i)
\subset \tilde W_i$ and $D$ is precompact in $W_i$.  Such charts may be
assumed to be volume preserving by \cite {M}.

Enumerate the vertices. 
Apply Theorem \ref {W} in charts to smooth $f$ in a neighborhood of the
first vertex without changing $f$ in a neighborhood of
simplices that do not contain this vertex.  Repeat with the subsequent
vertices.  Now suppose we have already
smoothed $f$ in a neighborhood of $m$-simplices, for some $0 \leq m \leq n-1$.
Enumerate the $m+1$-simplices and apply Theorem \ref {W} in
charts to smooth it in a neighborhood of the first $m+1$-simplex,
without changing it in a neighborhood of simplices that do not contain it
(in particular we do not change it near its boundary).
Repeat with the subsequent $m+1$-simplices.  After $n$ steps we will have
smoothed $f$ on the whole $M$.